\documentclass{amsart}  

\title{Meromorphic continuation of
  the Goldbach generating function}

\author{Gautami Bhowmik and Jan-Christoph Schlage-Puchta}

\newtheorem{theo}{Theorem}

\newtheorem{lemm}{Lemma}
\newtheorem{conj}{Conjecture}
\newtheorem{Cor}{Corollary}
\newcommand{\R}{\mathbb{R}}
\newcommand{\C}{\mathbb{C}}
\begin{document}
\begin{abstract}
We consider the Dirichlet series associated to the number of representations
of an integer as a sum of  primes. Assuming certain reasonable hypotheses on the distribution of the zeros of the Riemann zeta function we obtain
 the domain of
meromorphic continuation of this series.
\end{abstract}

\maketitle

\section{Introduction and Results}
In this paper we consider the number $G_r(n)$ of
representations of an integer $n$ as the sum of $r$ primes. One possible way to obtain
information is the use of complex integration. To do so Egami and
Matsumoto\cite{EgMat} introduced the generating function
\[
\Phi_r(s) = \sum_{k_1=1}^\infty\dots\sum_{k_r=1}^ \infty 
\frac{\Lambda(k_1)\dots\Lambda(k_r)}{(k_1+k_2+\dots+k_r)^s} 
 = \sum_{n=1}^\infty\frac{G_r(n)}{n^s}.
\]
This series is absolutely convergent for $\Re\;s>r$, and has a simple pole
at $s=r$. 
By Perron's formula we have
\[
\sum_{n\leq x} G_r(n) = \frac{1}{2\pi i}\int_{c-iT}^{c+iT}
\Phi_r(s)\frac{x^s}{s}\;ds + \mathcal{O}(\frac{x^{r+\epsilon}}{T}).
\]
To shift the path of integration to the left, one needs at least
meromorphic continuation to some half-plane $\Re\;s >r-\delta$ as well
as some information on the growth and the distribution of the poles of
$\Phi_r$. Assuming the Riemann hypothesis, Egami and
Matsumoto\cite{EgMat} described the behavior for the case $r=2$. In addition
to the RH, parts of their results depend on unproved assumptions on the
distribution of the imaginary parts of zeros of $\zeta$. 
Denote by $\Gamma$ the set of imaginary parts of non-trivial
zeros of $\zeta$. While the assumption that the positive elements in
$\Gamma$ are rationally independent appears to be folklore,
Fujii\cite{Fu1} drew  attention to the following special case:
\begin{conj}
\label{Con:Indep}
Suppose that $\gamma_1+\gamma_2=\gamma_3+\gamma_4\neq 0$ with
$\gamma_i\in\Gamma$. Then $\{\gamma_1, \gamma_2\}=\{\gamma_3,
\gamma_4\}$.
\end{conj}
Egami and Matsumoto used an effective version of this
conjecture, i.e. 
\begin{conj}
\label{Con:IndepEff}
There is some $\alpha<\frac{\pi}{2}$, such that for $\gamma_1, \ldots,
\gamma_4\in\Gamma$ we have either $\{\gamma_1,
\gamma_2\}=\{\gamma_3,\gamma_4\}$, or 
\[
|(\gamma_1+\gamma_2)-(\gamma_3+\gamma_4)|\geq\exp\big(-\alpha
(|\gamma_1| + |\gamma_2| + |\gamma_3| + |\gamma_4|)\big).
\]
\end{conj}
Obviously, Conjecture~\ref{Con:IndepEff} implies
Conjecture~\ref{Con:Indep}. In  \cite{EgMat}  it is proven that:
\begin{theo}
Suppose the Riemann hypothesis holds true. Then $\Phi_2(s)$ can be
meromorphically continued into the half-plane $\Re s>1$ with an
infinitude of poles on the line $\frac{3}{2}+it$. If in
addition Conjecture~\ref{Con:IndepEff} holds true, then the line
$\Re\;s=1$ is the natural boundary of $\Phi_2$. More precisely, the
set of points $1+i\kappa$ with $\lim_{\sigma\searrow 1}
|\Phi_2(\sigma+\kappa)| =\infty$ is dense on $\R$.
\end{theo}

The above mentioned authors conjectured that under the same
assumptions the domain of meromorphic continuation of $\Phi_r$ should
be the half-plane $\Re s>r-1$. In this direction we show that

\begin{theo}
\label{thm:Boundary}
Let $\mathcal{D}_r\subseteq\C$ be the domain of meromorphic continuation of
$\Phi_r(s)$.
\begin{enumerate}
\item If the RH holds true, then $\Phi_r(s)$ has a natural boundary at
  $\Re s=r-1$ for all $r\geq 2$ if and only if $\Phi_2(s)$ has a
  natural boundary at $\Re s=1$.
\item If the RH and Conjecture~\ref{Con:Indep} hold true, then
  $\Phi_2(s)$ has a natural boundary at $\Re s=1$.
\item If the RH holds true, then $\Phi_2$ has a singularity at $2\rho_1$, where
$\rho_1=\frac{1}{2}+14.1347\ldots i$ is the first root of $\zeta$. Moreover,
\begin{equation}
\label{eq:pole}
\lim_{\sigma_\searrow 0} (\sigma-1)|\phi_2(2\rho_1+\sigma)|>0.
\end{equation}
\item If Conjecture~\ref{Con:IndepEff} holds true, then
$\{s:\Re s>2\sigma_0\}\subseteq\mathcal{D}_2\subseteq\{s:\Re s>1\}$, where
$\sigma_0$ is the infimum over all real numbers $\sigma$, such that
$\zeta$ has only finitely many zeros in the half-plane $\Re
s>\sigma_0$.
\end{enumerate}
\end{theo}
The existence of a natural boundary already implies an $\Omega$-theorem
(confer \cite{natbd}). Here we can do a little better because of
(\ref{eq:pole}). We set
$$H_r(x)=-r\sum_{\rho}\frac{x^{r-1+\rho}}{\rho(1+\rho)\dots (r-1+\rho)}$$
 where the summation  
runs over all non-trivial zeros of $\zeta$ and we obtain the following
corollary.

\begin{Cor}
Suppose that RH holds true. Then we have
\[
\sum_{n\leq x} G_r(n) = \frac{1}{r!}x^r + H_r(x) +
\Omega(x^{r-1}).
\]
\end{Cor}
We note that without (\ref{eq:pole}) the omega term would still be
$\Omega(x^{r-1-\epsilon})$.
The corresponding $\mathcal{O}$-result  for $r=2$ is \cite{Nagoya} 
\begin{theo}
\label{thm:ErrorG2}
Suppose that the RH is true. Then
we have 
\[
\sum_{n\leq x} G_2(n) = \frac{1}{2}x^{2} + H_2(x) +
\mathcal{O}(x\log^5 x).
\]
\end{theo} 
It extends easily to $r\geq 3$. 
One might expect that
the quality of the error term would improve with $r$ increasing,
however this is not the case.

 Part (1) of
Theorem~\ref {thm:Boundary} follows immediately from the  assertion that the analytic behavior of
$\Phi_r$ is completely determined by the behavior of $\Phi_2$. More precisely

\begin{theo}
\label{thm:Rep}
Suppose the RH. Then for any $r\geq 3$ there exist rational functions
$f_{1,r}, \ldots, f_{4,r}(s)$, such that
\begin{multline*}
\Phi_r(s) = f_{1,r}(s)\zeta(s-r+1) + f_{2,r}(s)\zeta(s-r+2)\\
+f_{3,r}(s)\frac{\zeta'}{\zeta}(s-r+1) + f_{4,r}(s)\Phi_2(s-r+2) + R(s),
\end{multline*}
where $R(s)$ is holomorphic in the half-plane $\Re s>r-1-1/10$ and
uniformly bounded in each half-strip of the form $\Re s>r-1-1/10+\epsilon$, $T<\Im
s<T+1$, with $T>0$.
\end{theo}

The constant $1/10$ can be improved, however, since we believe that
$\Phi_2$ has $\Re\;s=1$ as natural boundary, we saw no point in doing so.

Our proof expresses the function $\Phi_r(s)$ using the circle method. 
This approach is the main novelty of this paper.

\par
\section{Proof of Theorem~\ref{thm:Rep}}

In this section we prove Theorem~\ref{thm:Rep} by computing the function
using the circle method. We use the standard notation.

Fix a large integer $x$, set $e(\alpha)=e^{2\pi i \alpha}$,
\begin{eqnarray*}
S(\alpha) & = & \sum_{n\leq x}\Lambda(n) e(\alpha n),\\
T(\alpha) & = & \sum_{n\leq x} e(\alpha n),\\
T_3(\alpha) & = & \sum_{|n|\leq x}(x-|n|)^2 e(n\alpha),\\
R(\alpha) & = & S(\alpha)-T(\alpha).
\end{eqnarray*}

\begin{lemm}
\label{Lem:ExpSum}
Under the Riemann hypothesis we have
\[
R(\alpha)\ll x^{1/2}\log^2 x + \alpha x^{3/2}\log^2 x. 
\]
\end{lemm}
\begin{proof}
The Riemann hypothesis is equivalent to the estimate
$\Psi(x)=x+\mathcal{O}(x^{1/2}\log^2 x)$, hence, 
\begin{eqnarray*}
R(\alpha) & = & \sum_{n\leq x}(\Lambda(n)-1)e(\alpha n)\\
 & = & (\Psi(x)-x)e(\alpha x) - \sum_{n\leq x}(\Psi(n)-n)(e(\alpha
 n+\alpha)-e(\alpha n))\\
 & \ll & x^{1/2}\log^2 x + \alpha x^{3/2}\log^2 x,
\end{eqnarray*}
and our claim follows.
\end{proof}
The next statement is a consequence of partial summation.
\begin{lemm}
\label{Lem:DPartial}
Let $a_n$ be a sequence of complex numbers, set $A_n=\sum_{\nu\leq n}
a_n$, $d(s)=\sum_n a_n n^{-s}$, and $D(s)=\sum_n A_n n^{-s}$. Suppose
that $D(s)$ is absolutely convergent for $\Re\;s>\sigma_0$ and has
meromorphic continuation to $\Re\;s>\sigma_1$. Then $d(s)$ has
meromorphic continuation to $\Re\;s>\sigma_1-1$, and there exist
polynomials $Q_i$, $0\leq i\leq \sigma_0-\sigma_1$, such that
\[
d(s) = \sum_{i=0}^{\lfloor\sigma_0-\sigma_1\rfloor} Q_i(s) D(s+1+i) + R(s),
\]
where $R$ is holomorphic on $\Re\;s>\sigma_1-1$, and continuous on
$\Re\;s\geq\sigma_1-1$. 
\end{lemm}
\begin{proof}
We have
\begin{eqnarray*}
d(s) & = & \sum_n A_n \big(n^{-s}-(n+1)^{-s}\big)\\
 & = & \sum_n A_n
\sum_{\nu=1}^N \frac{(-1)^{\nu+1} s(s+1)\dots(s+\nu-1)}{\nu!}
n^{-s-\nu} + R(s)\\
 & = & \sum_{\nu=1}^N \frac{(-1)^{\nu+1} s(s+1)\dots(s+\nu-1)}{\nu!}
 D(s+\nu) + R(s),
\end{eqnarray*}
where $R(s)$ is holomorphic on $\Re\;s>\sigma_0-N$. Choosing
$N>\sigma_0-\sigma_1+1$ our claim follows.
\end{proof}
We can now establish Theorem~\ref{thm:Rep}.
\begin{proof}[Proof of Theorem~\ref{thm:Rep}.] Define the sequence of functions
$A_r^k$ by $A_r^0(n)=G_r(n)$, and $A_r^{k+1}(n)=\sum_{\nu\leq n}
A_r^k(\nu)$.

We compute $A_r^3(x)$ using the circle method. 
We have
\begin{eqnarray*}
A_r^3(x) & = & \int\limits_0^1 S^r(\alpha) T_3(\alpha)d\alpha\\
 & = & \sum_{k=0}^r \binom{r}{k}\int\limits_0^1 T(\alpha)^{r-k}R^k(\alpha)
 T_3(\alpha)d\alpha\\
 & = & \sum_{k=0}^r \binom{r}{k} B_{r, k}(x),
\end{eqnarray*}
say. Our aim is to show that $B_{r, 0}(x), B_{r, 1}(x), B_{r, 2}(x)$
are quite regular and have main terms corresponding to the
Dirichlet-series explicitly mentioned in Theorem~\ref{thm:Rep}, and that
$B_{r, k}(x)$ for $k\geq 3$ is of order $\mathcal{O}(x^{r-1/10})$. We
collect the contribution of the coefficients $B_{r, k}(x)$ into a
Dirichlet-series, which converges uniformly in any half-plane of the
form $\Re\;s>r+1-1/10+\epsilon$, which can be absorbed into $R(s)$.
Once we have
shown these facts, Theorem~\ref{thm:Rep} follows. 

We first show that terms with $k\geq 3$ are negligible. Note
that $T_3(\alpha)\ll\min(x^3, \alpha^{-3})$. We split the integral
into the range $[-\beta, \beta]$ and $[\beta, 1-\beta]$. In the former
range, we use Lemma~\ref{Lem:ExpSum} to bound all occurring values of
$R$, whereas in the latter we use the estimate $\int_0^1
|R(\alpha)|^2\;d\alpha\ll x\log^2 x$. By symmetry it suffices to
consider the integral over $[0, 1/2]$, we begin with the case of small
$\alpha$. We have
\begin{eqnarray*}
\int\limits_0^\beta |T(\alpha)|^{r-k}|R(\alpha)|^k|T_3(\alpha)|
d\alpha & \ll & \int\limits_0^\beta\min(x^{r+3-k},
\alpha^{k-3-r})(x^{1/2}+\alpha x^{3/2})^kx^\epsilon\; d\alpha\\
 & \ll & x^{r+2-k/2+\epsilon}+\int\limits_{x^{-1}}^\beta
\alpha^{2k-2-r} x^{3k/2+\epsilon}\; d\alpha\\
 & \ll & x^{r+2-k/2+\epsilon} +
\begin{cases}
 \beta^{2k-1-r} x^{3k/2+\epsilon}, & 2k-1-r> 0\\
x^{r+2-k/2}, & 2k-1-r\leq0
\end{cases}
\end{eqnarray*}
Using $k\geq 3$ we see that in the second case the integral is bounded
by $x^{r+1/2}$, which is sufficient. In the first case we take
$\beta=x^{-1/2}$ and obtain that the integral is bounded by
$x^{(k+r+1)/2}$, since $k\leq r$, this is also of order $x^{r+1/2}$,
and therefore admissible.

For the remainder of the integral we use the $L^2$-estimate
$\int_0^1|R(\alpha)|^2d\alpha\ll x^{1+\epsilon}$ and the trivial bound
$|R(\alpha)|\ll x^{1+\epsilon}$ and obtain
\[
\int\limits_{x^{-1/2}}^{1/2} |T(\alpha)|^{r-k}|R(\alpha)|^k|T_3(\alpha)|
d\alpha\ll x^{1+\epsilon} \max_{x^{-1/2}\leq\alpha\leq 1/2}
\alpha^{k-3-r}x^{k-2} = x^{(k+r+1)/2},
\]
Using $k\leq r$ again we see that this is also
$\mathcal{O}(x^{r+1/2+\epsilon})$. Hence, we find that the
Dirichlet-series with coefficients $B_{r, k}(n)$ converge absolutely
for $\sigma > r+3/2$.

Next, we explicitly compute the contribution of the terms $k\leq
2$. We have $B_{r, 0}(x)= \int_0^1 T(\alpha)^r T_3(\alpha)\;d\alpha$,
that is, 
\begin{eqnarray*}
B_{r, 0}(x) & = & \sum_{n\leq x} (x-n)\#\{n_1+ \dots + n_r = n\}\\
 & = & \sum_{n\leq x} (x-n)^2 \binom{n+r-1}{r-1}\\
 & = & P_r(x)
\end{eqnarray*}
for some polynomial $P_r$ of degree $r+2$. Hence, the Dirichlet-series
with coefficients $B_{r, 0}$ can be expressed as a linear combination
of the functions $\zeta(s), \zeta(s-1), \ldots, \zeta(s-r-2)$. 

The corresponding computations for $B_{r, 1}$ and $B_{r, 2}$ are
simplified by observing that
\[
\int\limits_0^1 T(\alpha)^{r-1}R(\alpha) T_3(\alpha)\;d\alpha =  \int\limits_0^1
T(\alpha)^{r-1}S(\alpha) T_3(\alpha)\;d\alpha
 - \int\limits_0^1
T(\alpha)^r T_3(\alpha)\;d\alpha
\]
and
\begin{multline*}
\int\limits_0^1 T(\alpha)^{r-2} R(\alpha)^2 T_3(\alpha)\;d\alpha = \int\limits_0^1
T(\alpha)^{r-2}S(\alpha)^2 T_3(\alpha)\;d\alpha\\ - 2\int\limits_0^1
T(\alpha)^{r-1}S(\alpha) T_3(\alpha)\;d\alpha + \int\limits_0^1
T(\alpha)^{r} T_3(\alpha)\;d\alpha. 
\end{multline*}
To evaluate these integrals we transform them back into counting
problems. We have
\begin{eqnarray*}
\int\limits_0^1 T(\alpha)^r T_3(\alpha)\;d\alpha & = & \underset{0\leq
  n_i\leq x,\;|m|<x}{\sum_{n_1+\dots+n_r+m=0}}(x-|m|)^2\\
= \sum_{0\leq m\leq x}(x-m)^2\binom{m+r-1}{r-1} = P(x),
\end{eqnarray*}
where $P$ is a polynomial of degree $r+2$. Hence, the generating
function of $\int\limits_0^1 T(\alpha)^r T_2(\alpha)\;d\alpha$ is a
linear combination of $\zeta(s), \zeta(s-1), \ldots, \zeta(s-r-2)$.

Similarly,
\begin{eqnarray*}
\int\limits_0^1 T(\alpha)^{r-1}S(\alpha) T_3(\alpha)\;d\alpha & = & 
\underset{0\leq n_i\leq x,\;|m|<x}{\sum_{n_1+\dots+n_r+m=0}}\Lambda(m)(x-|m|)^2\\
 & = & \sum_{0\leq m\leq x}(x-m)\Lambda(m)\binom{m+r-1}{r-1}\\
 & = & \sum_{0\leq m\leq x} \Lambda(m) P_1(m) + x\sum_{0\leq m\leq x}
\Lambda(m) P_2(m),
\end{eqnarray*}
where $P_1$ is a polynomial of degree $r+2$, and $P_2$ a polynomial of
degree $r+1$. The generating function of $\Lambda(m) P_1(m)$ is a
linear combination of $\frac{\zeta'}{\zeta}(s),
\frac{\zeta'}{\zeta}(s-1), \ldots, \frac{\zeta'}{\zeta}(s-r-1)$,
applying partial summation we find that the generating function of
$\sum_{0\leq m\leq x} \Lambda(m) P_1(m)$ is a linear combination with
rational coefficients plus a remainder, which is holomorphic in the
half-plane $\Re\;s>0$, the same argument applies to the second sum.

Finally,
\[
\int\limits_0^1 T(\alpha)^{r-2}S(\alpha)^2 T_3(\alpha)\;d\alpha = 
\underset{0\leq n_i\leq x,\;|m|<x}{\sum_{n_1+\dots+n_{r-1}+m=0}}G_2(n_{r-1})(x-|m|)
\]
and as for the previous integral we find that the generating function
with coefficients $\int\limits_0^1 T(\alpha)^{r-2}S(\alpha)^2
T_3(\alpha)\;d\alpha$ is a linear combination of $\Phi_2(s)$, \ldots,
$\Phi_2(s-r-1)$ with rational coefficients, plus a function which is
holomorphic in the half-plane $\Re\;s>0$.

Combining this
observation with Lemma \ref{Lem:DPartial} we find that $\Phi_r(s)$ can
be written as a linear combination of the functions $\zeta(s)$, \ldots,
$\zeta(s-r+1)$, $\frac{\zeta'}{\zeta}(s)$, \ldots,
$\frac{\zeta'}{\zeta}(s-r+1)$, $\Phi_2(s)$, \ldots, $\Phi_2(s-r+2)$ with
rational coefficients plus a remainder $R(s)$ which is holomorphic in the
half-plane $\Re\;s>r-3/2$. But among these functions only
$\zeta(s-r+2)$, $\zeta(s-r+1)$, $\frac{\zeta'}{\zeta}(s-r+2)$,
$\frac{\zeta'}{\zeta}(s-r+1)$,$\Phi_2(s-r+3)$ and $\Phi_2(s-r+2)$ are not
holomorphic in the half-plane $\Re\;s>r-2$, hence, all but this six
functions can be subsumed under $R$. Moreover, since we work under the
Riemann hypothesis, $\frac{\zeta'}{\zeta}(s-r+2)$ and $\Phi_2(s-r+3)$
are holomorphic in 
$\Re\;s>r-3/2$ with the exception of a pole at $s=r-1$, hence, we
can replace these functions by $\zeta(s-r+2)$. Hence, the claim of the
theorem follows.
\end{proof}

\section{Proof of Theorem~\ref{thm:Boundary}}

Part (1) of the theorem follows from Theorem~\ref{thm:Rep} because in the half
plane $\Re(s)> r-1-1/10$  only $\Phi_2(s-r+2)$ has essential singularities.

We now indicate the proof of part (4) which is closely related 
to the one given under the Riemann
hypothesis by Egami and Matsumoto. The following serves as a
substitute for \cite[Lemma~4.1]{EgMat}
\begin{lemm}
let $D$ be the closure of the set $\{\rho_1+\rho_2:\zeta(\rho_i)=0,
\Re\rho_i>0\}$. Then $\C\setminus D$ is not connected, and, denoting
by $\mathcal{D}$ the component containing the half-plane $\Re s>2$, we have
\[
\{s:\Re s>3/2\}\subseteq\mathcal{D}\subseteq\{s:\Re s>1\}.
\]
\end{lemm}
\begin{proof}
Let $\epsilon>0$ and $t_0\in\R$ be given. We show that there are zeros
$\rho_1, \rho_2$ of $\zeta$ such that $\rho_1+\rho_2$ is within the
square $1\leq\Re s<1+\epsilon$, $t_0<\Im
s<t_0+\epsilon$, which implies the claim.
Let $N(T, \sigma)$ be the number of zeros $\rho$ of $\zeta$ with
$\Re\rho>\sigma$ and $0<\Im\rho<T$. Call a real number $t$ good, if
there is a zero $\rho$ of $\zeta$ with $\frac{1}{2}<\leq\Re
s<\frac{1}{2}+\epsilon/2$, $t_0<\Im s<t_0+\epsilon/2$, and let
$\mathcal{T}$ be the set of good numbers. We have
to show that there exists good numbers $t_1, t_2$ with
$t_1+t_2=t_0$. This in turn would follow if we show that
asymptotically almost all real numbers are good. To do so, we use the
estimate $N(T, \sigma)\ll T^{\frac{3(1-\sigma)}{2-\sigma}}\log^5 T$
due to Ingham and the fact that the distance between consecutive
abscissae of zeros tends to zero, proven by Littlewood. The second
statement shows that every sufficiently large real number $t$ is good,
unless there is a zero of $\zeta$ in the domain $\Re
s>\frac{1}{2}+\epsilon/2$, $t<\Im s<t+\epsilon/2$. Hence, we obtain
\[
|\mathcal{T}\cap[0, T]|\geq T-C(\epsilon)- N(T,\frac{1+\epsilon}{2}) \sim T,
\]
that is, for $T$ sufficiently large the measure of $\mathcal{T}\cap[0,
T]$ supersedes $T/2$, hence, we find real numbers $t_1,
t_2\in\mathcal{T}$ with $t_1+t_2=t_0$.
\end{proof}

It follows from \cite[Lemma~4.2]{EgMat}, that under
Conjecture~\ref{Con:IndepEff} every complex number of the form
$\rho_1+\rho_2$, $\zeta(\rho_1)=\zeta(\rho_2)=0$ is a singularity of
$\Phi_2$. The proof of the fact that
$\Phi_2$ is meromorphic in the half plane $\Re s>2\sigma_0$ runs parallel to
the proof of \cite[Theorem~2.1]{EgMat} and need not be repeated here.
 
Finally we prove parts (2) and (3). 
 
Let $\rho_1, \rho_2$ be zeros of $\zeta$. Our aim is to show that
either there are zeros $\rho_3, \rho_4$ 
with $\rho_1+\rho_2-\rho_3-\rho_4=0$ and $|\Im\rho_3|+|\Im\rho_4| \leq
5(|\Im\rho_1|+|\Im\rho_2|)$, or
$|\Phi_2(\rho_1+\rho_2+\eta)|\gg\frac{1}{\eta}$ 
for $\eta\searrow 0$. Our proof starts similar to the proof by Egami
and Matsumoto (confer \cite[section 4]{EgMat}).
\begin{lemm}
\label{Lem:Phi2Rep}
Put $M(s)=-\frac{\zeta'}{\zeta}(s)$. Then we have
\begin{eqnarray*}
\Phi_2(s) & = & \frac{M(s-1)}{s-1}-\sum_\rho\frac{\Gamma(s-\rho)\Gamma(\rho)}
{\Gamma(s)}M(s-\rho) - M(s)\log 2\pi\\
&&\qquad+\frac{1}{2\pi i}\int\limits_{-\epsilon-i\infty}^{-\epsilon+i\infty}
\frac{\Gamma(s-z)\Gamma(z)}{\Gamma(s)}M(s-z)M(z)dz\\
 & = & -\frac{1}{\Gamma(s)}\sum_{\rho, \rho'}\frac{\Gamma(s+1-\rho)\Gamma(\rho)}{(s-\rho-\rho')\rho'}
 + R(z),
\end{eqnarray*}
where $R$ is meromorphic in the whole complex plane
\end{lemm}
\begin{proof}
This follows from \cite[(2.2)]{EgMat} and\cite[(4.2)]{EgMat}.
\end{proof}
Now suppose that $\rho_1=\frac{1}{2}+i\gamma_1$,
$\rho_2=\frac{1}{2}+i\gamma_2$ are zeros 
of $\zeta$.  We want to show that in a small neighourhood of
$1+i(\gamma_1+\gamma_2)$ the behaviour of $\Phi_2(s)$ is dominated by
the summand coming from $\rho_1, \rho_2$. To do so we estimate
the contribution of different ranges for $\rho, \rho'$ in different
ways.

Consider pairs $\rho, \rho'$ with
$|\rho+\rho'-\rho_1-\rho_2|>\frac{1}{4}$. For
$|s-\rho_1-\rho_2|<\frac{1}{8}$, the sum can be bounded as
\begin{eqnarray*}
\sum_{\rho,\rho'}\frac{\Gamma(s+1-\rho)\Gamma(\rho)}
{(s-\rho-\rho')\rho'} & = & \underset{|\rho_1+\rho_2-\rho-\rho'|>\rho'/2}
{\sum_{\rho, \rho'}}\frac{\Gamma(s+1-\rho)\Gamma(\rho)}
{(s-\rho-\rho')\rho'}\\
&&\qquad + \underset{|\rho_1+\rho_2-\rho-\rho'|\leq\rho'/2}
{\sum_{\rho, \rho'}}\frac{\Gamma(s+1-\rho)\Gamma(\rho)}
{(s-\rho-\rho')\rho'}\\
 & = & \sum\nolimits_1 + \sum\nolimits_2,
\end{eqnarray*}
say. We have
\[
\sum\nolimits_1 \ll \underset{|\rho_1+\rho_2-\rho-\rho'|>\rho'/2}
{\sum_{\rho, \rho'}}\frac{\Gamma(\rho)}
{\rho'^2} \ll \underset{|\rho_1+\rho_2-\rho-\rho'|>\rho'/2}
{\sum_{\rho, \rho'}}\frac{\Gamma(\rho)}
{\rho'^2}
\]
and
\[
\sum\nolimits_2\ll \underset{|\rho_1+\rho_2-\rho-\rho'|\leq\rho'/2}
{\sum_{\rho, \rho'}}\Gamma(s+1-\rho)\Gamma(\rho)\ll\sum_\rho\Gamma(\rho)N(2|\rho|),
\]
Since $N(T)\ll T\log T$, and $\Gamma(\sigma+it)\ll e^{-ct}$, we see
that this sum converges uniformly in the open ball
$B_{\frac{1}{8}}(\rho_1+\rho_2)$. 

Now consider pairs of zeros with $|\rho|+|\rho'|>5(|\rho_1|+|\rho_2|)$ and
$|\rho+\rho'-\rho_1-\rho_2|\leq 1/4$. For $s\in B_{\frac{1}{8}}(\rho_1+\rho_2)$
we have 
\[
s+1-\rho=1+\rho_1+\rho_2-\rho+\theta/8=1+\rho'+3\theta/8,
\]
where $\theta$ is a complex number of absolute value $\leq 1$, hence,
the sum taken 
over all $\rho, \rho'$ in this range is bounded by 
\begin{eqnarray*}
\sum_{\rho, \rho'} \left|\frac{\Gamma(\rho)
\max_{|s-\rho'|\leq 3/8}|\Gamma(1+s)|}{(s-\rho-\rho')\rho'}\right| & <
& \frac{1}{\eta}\underset{|\rho+\rho'-\rho_1-\rho_2|<1}
{\sum_{|\rho|+|\rho'|>5(|\rho_1|+|\rho_2|)}}|\rho'\Gamma(\rho)\Gamma(\rho')|\\
& < & \frac{1}{\eta}\left(\sum_{\rho>2(|\rho_1|+|\rho_2|)-1}
|\rho\Gamma(\rho)|\right)^2.
\end{eqnarray*}
To transform the sum over zeros into a sum over integers, we use the
following bound, which follows from a more precise result by
Backlund\cite{Back}.
\begin{lemm}
\label{Lem:Backlund}
We have $N(T+1)-N(T)<\log T$ for $T\geq 1$.
\end{lemm}
Using this bound together with the estimate $\Gamma(\sigma+it) <
e^{-\frac{\pi}{4}t}$ we obtain that the contribution of zeros of the
form under consideration is bounded by
\[
\frac{1}{\eta}\left(\sum_{n\geq 2(|\rho_1|+|\rho_2|)-2} n\log n
  e^{-\frac{\pi}{4}n}\right)^2 < \frac{270}{\eta}e^{-\frac{12}{5}(|\rho_1|+|\rho_2|)},
\]
where we used the fact that $|\rho_1|+|\rho_2|>28$.

On the other hand, the pair $\rho_1, \rho_2$ itself contributes
\[
\frac{1}{\eta}\left(\frac{\Gamma(\rho_1+1+\eta)\Gamma(\rho_2)}{\rho_1}
  + \frac{\Gamma(\rho_2+1+\eta)\Gamma(\rho_1)}{\rho_2}\right)\sim \frac{2}{\eta}\Gamma(\rho_1)\Gamma(\rho_2),
\]
hence, the contribution of these zeros cannot be canceled by the
contribution of  zeros satisfying $|\rho|+|\rho'|>5(|\rho_1|+|\rho_2|)$  
since
\[
2e^{-\frac{4}{5}(|\rho_1|+|\rho_2|)} > 270e^{-\frac{12}{5}(|\rho_1|+|\rho_2|)}.
\]
In fact the above inequality follows from $|\rho_1|+|\rho_2|>28$.

The arguments used till now were under the assumption of RH. Now we separate the
arguments for parts (2) and (3).

If in addition to the RH we
assume Conjecture 1, then the
finitely many pairs $\rho, \rho'$ different from $\rho_1, \rho_2$ which
we have not yet dealt with define a function meromorphic on $\C$ without poles
on the line $\Im s=\gamma_1+\gamma_2$.  In some neighborhood of
$\rho_1+\rho_2$ this function is bounded and the proof of part (2) is done.

We did not actually use the full strength of Conjecture 1 but only the
non-existence of linear relations $\gamma_1+\gamma_2=\gamma_3+\gamma_4\neq 0$
for $|\gamma_1|+|\gamma_2|\le 5(|\gamma_3|+|\gamma_4|)$. Even though this
condition seems as unreachable as Conjecture 1, for each fixed pair
$\gamma_1,\gamma_2$ it can be verified. Taking the minimal value 14.1347\ldots 
for $\gamma_1,\gamma_2$ there are only 39 zeros with imaginary part at most 142
it is easy to check that no other pair adds up to $\gamma_1+\gamma_2$.
hence $2\rho_1$ is a singularity of $\Phi_2$ and we are done. 
 
We now prove the Corollary. Set
\[
\Delta_r(x) = \sum_{n\leq x} G_r(n) - \frac{1}{r!}x^r - H_r(x).
\]
In view of Lemma~\ref{Lem:DPartial} and
Theorem~\ref{thm:Boundary}, part (3) the generating Dirichlet-series $D(s)$ of $\Delta_r$ has a singularity at
$2\rho_1+r-1$ . Moreover, 
\[
\lim_{\sigma\searrow 0} \sigma |D(\sigma+2\rho_1+r-1)|>0
\]
and if we have $\Delta_r(x)=o(x^{r-1})$, the last limit is zero. Thus our claim
follows.

\end{document}